\documentclass{amsart}
\usepackage{amssymb}
\usepackage{amsfonts}

\newtheorem{theorem}{Theorem}
\theoremstyle{plain}

\newtheorem{corollary}{Corollary}

\newtheorem{definition}{Definition}

\newtheorem{lemma}{Lemma}

\newtheorem{remark}{Remark}

\numberwithin{equation}{section}
\input{tcilatex}

\begin{document}
\title[Inequalities via preinvexity and prequasiinvexity]{Simpson type
inequalities for first order differentiable preinvex and prequasiinvex
functions}
\author{M.Emin \"{O}zdemir$^{\blacklozenge }$}
\address{$^{\blacklozenge }$Atat\"{u}rk University, K.K. Education Faculty,
Department of Mathematics, Erzurum 25240, Turkey}
\email{emos@atauni.edu.tr}
\author{Merve Avci Ardic$^{\bigstar \diamondsuit }$}
\address{$^{\bigstar }$Adiyaman University, Faculty of Science and Arts,
Department of Mathematics, Adiyaman 02040, Turkey}
\email{mavci@adiyaman.edu.tr}
\thanks{$^{\diamondsuit }$Corresponding Author}

\begin{abstract}
In this paper, we obtain some inequalities for functions whose first
derivatives in absolute value are preinvex and prequasiinvex.
\end{abstract}

\maketitle

\section{introduction and preliminaries}

Suppose $f:[a,b]\rightarrow 
\mathbb{R}
$ is a four times continuously differentiable mapping on $(a,b)$ and $%
\left\Vert f^{(4)}\right\Vert _{\infty }=\sup \left\vert
f^{(4)}(x)\right\vert <\infty .$ The following inequality%
\begin{eqnarray*}
&&\left\vert \frac{1}{3}\left[ \frac{f(a)+f(b)}{2}+2f\left( \frac{a+b}{2}%
\right) \right] -\frac{1}{b-a}\int_{a}^{b}f(x)dx\right\vert \\
&\leq &\frac{1}{2880}\left\Vert f^{(4)}\right\Vert _{\infty }\left(
b-a\right) ^{4}
\end{eqnarray*}%
is well known in the literature as Simpson's inequality.

For some results about Simpson inequality see \cite{ADS}-\cite{D}.

Let $K$ be a nonempty closed set in $%
\mathbb{R}
^{n}.$ We denote by $\langle .,.\rangle $ and $\left\Vert .\right\Vert $ the
inner product and norm respectively. Let $f:K\rightarrow 
\mathbb{R}
$ and $\eta :K\times K\rightarrow 
\mathbb{R}
$ be continuous functions.

\begin{definition}
(See \cite{12}) Let $u\in K.$ Then the set $K$ is said to be invex at $u$
with respect to $\eta \left( .,.\right) ,$ if 
\begin{equation*}
u+t\eta (v,u)\in K,\text{ \ \ }\forall u,v\in K,\text{ \ \ }t\in \lbrack
0,1].
\end{equation*}%
$K$ is said to be invex set with respect to $\eta ,$ if $K$ is invex at each 
$u\in K.$ The invex set $K$ is also called a $\eta -$connected set.
\end{definition}

\begin{remark}
(See \cite{11}) We would like to mention that the Definition 1 of an invex
set has a clear geometric interpretation. This definition essentially says
that there is a path starting from a point $u$ which is contained in $K.$ We
don't require that the point $v$ should be one of the end points of the
path. This observation plays an important role in our analysis . Note that,
if we demand that $v$ should be an end point of the path for every pair of
points, $u,v\in K,$ then $\eta (v,u)=v-u$ and consequently invexity reduces
to convexity. Thus, it is true that every convex set is also an invex set
with respect to $\eta (v,u)=v-u$, but the converse is not necessarily true.
\end{remark}

\begin{definition}
(See \cite{12}) The function $f$ on the invex set $K$ is said to be preinvex
with respect to $\eta ,$ if 
\begin{equation*}
f(u+t\eta (v,u))\leq (1-t)f(u)+tf(v),\text{ \ \ }\forall u,v\in K,\text{ \ \ 
}t\in \lbrack 0,1].
\end{equation*}%
The function $f$ is said to be preconcave if and only if $-f$ is preinvex.
Note that every convex function is a preinvex function, but the converse is
not true. For example, the function $f(u)=-\left\vert u\right\vert $ is not
a convex function, but it is a preinvex function with respect to $\eta ,$
where 
\begin{equation*}
\eta (v,u)=\left\{ 
\begin{array}{ccc}
v-u,\text{ } &  & \text{if }v\leq 0,u\leq 0\text{ \ \ and \ \ }v\geq 0,u\geq
0 \\ 
&  &  \\ 
u-v, &  & \text{\ otherwise.}%
\end{array}%
\right.
\end{equation*}
\end{definition}

\begin{definition}
(See \cite{10}) The function $f$ on the invex set $K$ is said to be
prequasiinvex with respect to $\eta ,$ if%
\begin{equation*}
f(u+t\eta (v,u))\leq \max \left\{ f(u),f(v)\right\} ,\text{ \ \ }\forall
u,v\in K,\text{ \ \ }t\in \lbrack 0,1].
\end{equation*}
\end{definition}

In this paper, we establish some new Simpson type inequalities for preinvex
and prequasiinvex functions.

\section{Simpson type inequalities for preinvex functions}

We used the following Lemma to obtain our main results.

\begin{lemma}
\label{lem 3.1} Let $I\subseteq 
\mathbb{R}
$ be an open invex subset with respect to $\eta :I\times I\rightarrow 
\mathbb{R}
_{+}$ and $f:I\rightarrow 
\mathbb{R}
$ be an absolutely continuous mapping on $I$,$a,b\in I$ with $\eta \left(
b,a\right) \neq 0.$ If $f^{\prime }$ is integrable on $\eta -$path $P_{ac},$ 
$c=a+\eta (b,a),$ following equality holds:%
\begin{eqnarray*}
&&\left\vert \frac{1}{6}\left[ f(a)+4f\left( \frac{2a+\eta (b,a)}{2}\right)
+f(a+\eta (b,a))\right] -\frac{1}{\eta (b,a)}\int_{a}^{a+\eta
(b,a)}f(x)dx\right\vert  \\
&=&\eta (b,a)\int_{0}^{1}m(t)f^{\prime }(a+t\eta (b,a))dt,
\end{eqnarray*}%
where%
\begin{equation*}
\ \ m(t)=\left\{ 
\begin{array}{c}
t-\frac{1}{6},\text{ \ \ \ \ }t\in \left[ 0,\frac{1}{2}\right)  \\ 
\\ 
t-\frac{5}{6},\text{ \ \ \ \ }t\in \left[ \frac{1}{2},1\right] .%
\end{array}%
\right. 
\end{equation*}
\end{lemma}

\begin{proof}
Since $a,b\in I$ and $I$ is an invex set with respect to $\eta ,$ it is
obvious that $a+t\eta \left( b,a\right) \in I$ for $t\in \left[ 0,1\right] $
Integrating by parts implies that%
\begin{eqnarray*}
&&\int_{0}^{\frac{1}{2}}\left( t-\frac{1}{6}\right) f^{\prime }(a+t\eta
(b,a))dt+\int_{\frac{1}{2}}^{1}\left( t-\frac{5}{6}\right) f^{\prime
}(a+t\eta (b,a))dt \\
&=&\left. \left( t-\frac{1}{6}\right) \frac{f(a+t\eta (b,a))}{\eta (b,a)}%
\right\vert _{0}^{\frac{1}{2}}-\int_{0}^{\frac{1}{2}}\frac{f(a+t\eta (b,a))}{%
\eta (b,a)}dt \\
&&+\left. \left( t-\frac{5}{6}\right) \frac{f(a+t\eta (b,a))}{\eta (b,a)}%
\right\vert _{\frac{1}{2}}^{1}-\int_{\frac{1}{2}}^{1}\frac{f(a+t\eta (b,a))}{%
\eta (b,a)}dt \\
&=&\frac{1}{6\eta (b,a)}\left[ f(a)+4f\left( \frac{2a+\eta (b,a)}{2}\right)
+f(a+\eta (b,a))\right] \\
&&-\frac{1}{\eta (b,a)}\left[ \int_{0}^{\frac{1}{2}}f(a+t\eta (b,a))dt+\int_{%
\frac{1}{2}}^{1}f(a+t\eta (b,a))dt\right] .
\end{eqnarray*}%
If we change the variable $x=a+t\eta (b,a)$ and multiply the resulting
equality with $\eta (b,a)$ we get the desired result.
\end{proof}

\begin{theorem}
\label{teo 3.1} Let $I\subseteq 
\mathbb{R}
$ be an open invex subset with respect to $\eta :I\times I\rightarrow 
\mathbb{R}
_{+}$ and $f:I\rightarrow 
\mathbb{R}
$ be an absolutely continuous mapping on $I$,$a,b\in I$ with $\eta \left(
b,a\right) \neq 0.$ If $\left\vert f^{\prime }\right\vert $ is preinvex then
the following inequality holds:%
\begin{eqnarray*}
&&\left\vert \frac{1}{6}\left[ f(a)+4f\left( \frac{2a+\eta (b,a)}{2}\right)
+f(a+\eta (b,a))\right] -\frac{1}{\eta (b,a)}\int_{a}^{a+\eta
(b,a)}f(x)dx\right\vert  \\
&\leq &\frac{5}{72}\eta (b,a)\left[ \left\vert f^{\prime }(a)\right\vert
+\left\vert f^{\prime }(b)\right\vert \right] .
\end{eqnarray*}
\end{theorem}

\begin{proof}
From Lemma \ref{lem 3.1} and using the preinvexity of $\left\vert f^{\prime
}\right\vert $ we have%
\begin{eqnarray*}
&&\left\vert \frac{1}{6}\left[ f(a)+4f\left( \frac{2a+\eta (b,a)}{2}\right)
+f(a+\eta (b,a))\right] -\frac{1}{\eta (b,a)}\int_{a}^{a+\eta
(b,a)}f(x)dx\right\vert \\
&\leq &\eta (b,a)\left\{ \int_{0}^{\frac{1}{2}}\left\vert t-\frac{1}{6}%
\right\vert \left\vert f^{\prime }(a+t\eta (b,a))\right\vert dt+\int_{\frac{1%
}{2}}^{1}\left\vert t-\frac{5}{6}\right\vert f^{\prime }(a+t\eta
(b,a))dt\right\} \\
&\leq &\eta (b,a)\left\{ \int_{0}^{\frac{1}{6}}\left( \frac{1}{6}-t\right) %
\left[ \left( 1-t\right) \left\vert f^{\prime }(a)\right\vert +t\left\vert
f^{\prime }(b)\right\vert \right] dt\right. \\
&&+\left. \int_{\frac{1}{6}}^{\frac{1}{2}}\left( t-\frac{1}{6}\right) \left[
\left( 1-t\right) \left\vert f^{\prime }(a)\right\vert +t\left\vert
f^{\prime }(b)\right\vert \right] dt\right. \\
&&+\left. \int_{\frac{1}{2}}^{\frac{5}{6}}\left( \frac{5}{6}-t\right) \left[
\left( 1-t\right) \left\vert f^{\prime }(a)\right\vert +t\left\vert
f^{\prime }(b)\right\vert \right] dt\right. \\
&&+\left. \int_{\frac{5}{6}}^{1}\left( t-\frac{5}{6}\right) \left[ \left(
1-t\right) \left\vert f^{\prime }(a)\right\vert +t\left\vert f^{\prime
}(b)\right\vert \right] dt\right. .
\end{eqnarray*}%
If we compute the above integrals, \ we get the desired result.
\end{proof}

\begin{theorem}
\label{teo 3.2} Let $I\subseteq 
\mathbb{R}
$ be an open invex subset with respect to $\eta :I\times I\rightarrow 
\mathbb{R}
_{+}$ and $f:I\rightarrow 
\mathbb{R}
$ be an absolutely continuous mapping on $I$,$a,b\in I$ with $\eta \left(
b,a\right) \neq 0.$ If $\left\vert f^{\prime }\right\vert $ is preinvex for
some fixed $q>1$ then the following inequality holds%
\begin{eqnarray*}
&&\left\vert \frac{1}{6}\left[ f(a)+4f\left( \frac{2a+\eta (b,a)}{2}\right)
+f(a+\eta (b,a))\right] -\frac{1}{\eta (b,a)}\int_{a}^{a+\eta
(b,a)}f(x)dx\right\vert  \\
&\leq &\eta \left( b,a\right) \left( \frac{1+2^{p+1}}{6^{p+1}(p+1)}\right) ^{%
\frac{1}{p}} \\
&&\times \left\{ \left( \frac{3}{8}\left\vert f^{\prime }(a)\right\vert ^{q}+%
\frac{1}{8}\left\vert f^{\prime }(b)\right\vert ^{q}\right) ^{\frac{1}{q}%
}+\left( \frac{1}{8}\left\vert f^{\prime }(a)\right\vert ^{q}+\frac{3}{8}%
\left\vert f^{\prime }(b)\right\vert ^{q}\right) ^{\frac{1}{q}}\right\} 
\end{eqnarray*}%
where $p=\frac{q}{q-1}.$
\end{theorem}

\begin{proof}
From Lemma \ref{lem 3.1} and using the H\"{o}lder inequality, we have 
\begin{eqnarray*}
&&\left\vert \frac{1}{6}\left[ f(a)+4f\left( \frac{2a+\eta (b,a)}{2}\right)
+f(a+\eta (b,a))\right] -\frac{1}{\eta (b,a)}\int_{a}^{a+\eta
(b,a)}f(x)dx\right\vert \\
&\leq &\eta (b,a)\left\{ \left( \int_{0}^{\frac{1}{2}}\left\vert t-\frac{1}{6%
}\right\vert ^{p}dt\right) ^{\frac{1}{p}}\left( \int_{0}^{\frac{1}{2}%
}\left\vert f^{\prime }(a+t\eta (b,a))\right\vert ^{q}dt\right) ^{\frac{1}{q}%
}\right. \\
&&+\left. \left( \int_{\frac{1}{2}}^{1}\left\vert t-\frac{5}{6}\right\vert
^{p}dt\right) ^{\frac{1}{p}}\left( \int_{\frac{1}{2}}^{1}\left\vert
f^{\prime }(a+t\eta (b,a))\right\vert ^{q}dt\right) ^{\frac{1}{q}}\right\} .
\end{eqnarray*}%
Since $\left\vert f^{\prime }\right\vert ^{q}$ is preinvex, we obtain 
\begin{eqnarray*}
&&\left\vert \frac{1}{6}\left[ f(a)+4f\left( \frac{2a+\eta (b,a)}{2}\right)
+f(a+\eta (b,a))\right] -\frac{1}{\eta (b,a)}\int_{a}^{a+\eta
(b,a)}f(x)dx\right\vert \\
&\leq &\eta (b,a)\left\{ \left( \int_{0}^{\frac{1}{6}}\left( \frac{1}{6}%
-t\right) ^{p}dt+\int_{\frac{1}{6}}^{\frac{1}{2}}\left( t-\frac{1}{6}\right)
^{p}dt\right) ^{\frac{1}{p}}\right. \\
&&\times \left. \left( \int_{0}^{\frac{1}{2}}\left[ \left( 1-t\right)
\left\vert f^{\prime }(a)\right\vert ^{q}+t\left\vert f^{\prime
}(b)\right\vert ^{q}\right] dt\right) ^{\frac{1}{q}}\right. \\
&&+\left. \left( \int_{\frac{1}{2}}^{\frac{5}{6}}\left( \frac{5}{6}-t\right)
^{p}dt+\int_{\frac{5}{6}}^{1}\left( t-\frac{5}{6}\right) ^{p}dt\right) ^{%
\frac{1}{p}}\right. \\
&&\times \left. \left( \int_{\frac{1}{2}}^{1}\left[ \left( 1-t\right)
\left\vert f^{\prime }(a)\right\vert ^{q}+t\left\vert f^{\prime
}(b)\right\vert ^{q}\right] dt\right) ^{\frac{1}{q}}\right\} \\
&=&\eta (b,a)\left( \frac{1+2^{p+1}}{6^{p+1}(p+1)}\right) ^{\frac{1}{p}} \\
&&\times \left\{ \left( \frac{3}{8}\left\vert f^{\prime }(a)\right\vert ^{q}+%
\frac{1}{8}\left\vert f^{\prime }(b)\right\vert ^{q}\right) ^{\frac{1}{q}%
}+\left( \frac{1}{8}\left\vert f^{\prime }(a)\right\vert ^{q}+\frac{3}{8}%
\left\vert f^{\prime }(b)\right\vert ^{q}\right) ^{\frac{1}{q}}\right\} .
\end{eqnarray*}%
The proof is completed.
\end{proof}

\begin{theorem}
\label{teo 3.3} Under the assumptions of Theorem \ref{teo 3.2}, we have the
following inequality%
\begin{eqnarray*}
&&\left\vert \frac{1}{6}\left[ f(a)+4f\left( \frac{2a+\eta (b,a)}{2}\right)
+f(a+\eta (b,a))\right] -\frac{1}{\eta (b,a)}\int_{a}^{a+\eta
(b,a)}f(x)dx\right\vert \\
&\leq &\eta \left( b,a\right) \left( \frac{2(1+2^{p+1})}{6^{p+1}(p+1)}%
\right) ^{\frac{1}{p}}\left[ \frac{\left\vert f^{\prime }(a)\right\vert
^{q}+\left\vert f^{\prime }(b)\right\vert ^{q}}{2}\right] ^{\frac{1}{q}}.
\end{eqnarray*}
\end{theorem}

\begin{proof}
From Lemma \ref{lem 3.1}, preinvexity of $\left\vert f^{\prime }\right\vert
^{q}$ and using the H\"{o}lder inequality, we have 
\begin{eqnarray*}
&&\left\vert \frac{1}{6}\left[ f(a)+4f\left( \frac{2a+\eta (b,a)}{2}\right)
+f(a+\eta (b,a))\right] -\frac{1}{\eta (b,a)}\int_{a}^{a+\eta
(b,a)}f(x)dx\right\vert \\
&\leq &\eta (b,a)\left[ \int_{0}^{1}\left\vert m(t)\right\vert \left\vert
f^{\prime }(a+t\eta (b,a))\right\vert dt\right] \\
&\leq &\eta (b,a)\left( \int_{0}^{1}\left\vert m(t)\right\vert ^{p}dt\right)
^{\frac{1}{p}}\left( \int_{0}^{1}\left\vert f^{\prime }(a+t\eta
(b,a))\right\vert ^{q}dt\right) ^{\frac{1}{q}} \\
&\leq &\eta (b,a)\left( \int_{0}^{\frac{1}{2}}\left\vert t-\frac{1}{6}%
\right\vert ^{p}dt+\int_{\frac{1}{2}}^{1}\left\vert t-\frac{5}{6}\right\vert
^{p}dt\right) ^{\frac{1}{p}}\left( \int_{0}^{1}\left[ \left( 1-t\right)
\left\vert f^{\prime }(a)\right\vert ^{q}+t\left\vert f^{\prime
}(b)\right\vert ^{q}\right] dt\right) ^{\frac{1}{q}} \\
&=&\eta (b,a)\left( \frac{2(1+2^{p+1})}{6^{p+1}(p+1)}\right) ^{\frac{1}{p}}%
\left[ \frac{\left\vert f^{\prime }(a)\right\vert ^{q}+\left\vert f^{\prime
}(b)\right\vert ^{q}}{2}\right] ^{\frac{1}{q}}
\end{eqnarray*}%
where we used the fact that%
\begin{equation*}
\int_{0}^{\frac{1}{2}}\left\vert t-\frac{1}{6}\right\vert ^{p}dt=\int_{\frac{%
1}{2}}^{1}\left\vert t-\frac{5}{6}\right\vert ^{p}dt=\frac{(1+2^{p+1})}{%
6^{p+1}(p+1)}.
\end{equation*}%
The proof is completed.
\end{proof}

\begin{theorem}
\label{teo 3.4} Let $I\subseteq 
\mathbb{R}
$ be an open invex subset with respect to $\eta :I\times I\rightarrow 
\mathbb{R}
_{+}$ and $f:I\rightarrow 
\mathbb{R}
$ be an absolutely continuous mapping on $I$,$a,b\in I$ with $\eta \left(
b,a\right) \neq 0.$ If $\left\vert f^{\prime }\right\vert $ is preinvex for
some fixed $q\geq 1$ then the following inequality holds%
\begin{eqnarray*}
&&\left\vert \frac{1}{6}\left[ f(a)+4f\left( \frac{2a+\eta (b,a)}{2}\right)
+f(a+\eta (b,a))\right] -\frac{1}{\eta (b,a)}\int_{a}^{a+\eta
(b,a)}f(x)dx\right\vert  \\
&\leq &\eta \left( b,a\right) \left( \frac{5}{72}\right) ^{1-\frac{1}{q}} \\
&&\times \left\{ \left( \frac{61\left\vert f^{\prime }(a)\right\vert
^{q}+29\left\vert f^{\prime }(b)\right\vert ^{q}}{1296}\right) ^{\frac{1}{q}%
}+\left( \frac{29\left\vert f^{\prime }(a)\right\vert ^{q}+61\left\vert
f^{\prime }(b)\right\vert ^{q}}{1296}\right) ^{\frac{1}{q}}\right\} .
\end{eqnarray*}
\end{theorem}

\begin{proof}
From Lemma \ref{lem 3.1} and using the power-mean inequality, we have%
\begin{eqnarray*}
&&\left\vert \frac{1}{6}\left[ f(a)+4f\left( \frac{2a+\eta (b,a)}{2}\right)
+f(a+\eta (b,a))\right] -\frac{1}{\eta (b,a)}\int_{a}^{a+\eta
(b,a)}f(x)dx\right\vert \\
&\leq &\eta \left( b,a\right) \\
&&\times \left\{ \left( \int_{0}^{\frac{1}{2}}\left\vert t-\frac{1}{6}%
\right\vert dt\right) ^{1-\frac{1}{q}}\left( \int_{0}^{\frac{1}{2}%
}\left\vert t-\frac{1}{6}\right\vert \left\vert f^{\prime }(a+t\eta \left(
b,a\right) )\right\vert ^{q}dt\right) ^{\frac{1}{q}}\right. \\
&&+\left. \left( \int_{\frac{1}{2}}^{1}\left\vert t-\frac{5}{6}\right\vert
dt\right) ^{1-\frac{1}{q}}\left( \int_{\frac{1}{2}}^{1}\left\vert t-\frac{5}{%
6}\right\vert \left\vert f^{\prime }(a+t\eta \left( b,a\right) )\right\vert
^{q}dt\right) ^{\frac{1}{q}}\right\} .
\end{eqnarray*}%
Since $\left\vert f^{\prime }\right\vert ^{q}$ is preinvex function we have%
\begin{eqnarray*}
&&\int_{0}^{\frac{1}{2}}\left\vert t-\frac{1}{6}\right\vert \left\vert
f^{\prime }(a+t\eta \left( b,a\right) )\right\vert ^{q}dt \\
&\leq &\int_{0}^{\frac{1}{6}}\left( \frac{1}{6}-t\right) \left[ \left(
1-t\right) \left\vert f^{\prime }(a)\right\vert ^{q}+t\left\vert f^{\prime
}(b)\right\vert ^{q}\right] dt \\
&&+\int_{\frac{1}{6}}^{\frac{1}{2}}\left( t-\frac{1}{6}\right) \left[ \left(
1-t\right) \left\vert f^{\prime }(a)\right\vert ^{q}+t\left\vert f^{\prime
}(b)\right\vert ^{q}\right] dt \\
&=&\frac{61\left\vert f^{\prime }(a)\right\vert ^{q}+29\left\vert f^{\prime
}(b)\right\vert ^{q}}{1296}
\end{eqnarray*}%
and 
\begin{eqnarray*}
&&\int_{\frac{1}{2}}^{1}\left\vert t-\frac{5}{6}\right\vert \left\vert
f^{\prime }(a+t\eta \left( b,a\right) )\right\vert ^{q}dt \\
&\leq &\int_{\frac{1}{2}}^{\frac{5}{6}}\left( \frac{5}{6}-t\right) \left[
\left( 1-t\right) \left\vert f^{\prime }(a)\right\vert ^{q}+t\left\vert
f^{\prime }(b)\right\vert ^{q}\right] dt \\
&&+\int_{\frac{5}{6}}^{1}\left( t-\frac{5}{6}\right) \left[ \left(
1-t\right) \left\vert f^{\prime }(a)\right\vert ^{q}+t\left\vert f^{\prime
}(b)\right\vert ^{q}\right] dt \\
&=&\frac{29\left\vert f^{\prime }(a)\right\vert ^{q}+61\left\vert f^{\prime
}(b)\right\vert ^{q}}{1296}.
\end{eqnarray*}%
Combining all the above inequalities gives us the desired result.
\end{proof}

\section{Simpson type inequalities for prequasiinvex functions}

In this section, we obtained Simpson type inequalities for prequasiinvex
functions.

\begin{theorem}
\label{teo 4.1} Let $I\subseteq 
\mathbb{R}
$ be an open invex subset with respect to $\eta :I\times I\rightarrow 
\mathbb{R}
_{+}$ and $f:I\rightarrow 
\mathbb{R}
$ be an absolutely continuous mapping on $I$,$a,b\in I$ with $\eta \left(
b,a\right) \neq 0.$ If $\left\vert f^{\prime }\right\vert $ is prequasiinvex
for some fixed $q\geq 1$ then the following inequality holds%
\begin{eqnarray*}
&&\left\vert \frac{1}{6}\left[ f(a)+4f\left( \frac{2a+\eta (b,a)}{2}\right)
+f(a+\eta (b,a))\right] -\frac{1}{\eta (b,a)}\int_{a}^{a+\eta
(b,a)}f(x)dx\right\vert  \\
&\leq &\frac{5}{36}\eta \left( b,a\right) \left[ \max \left\{ \left\vert
f^{\prime }(a)\right\vert ^{q},\left\vert f^{\prime }(b)\right\vert
^{q}\right\} \right] ^{\frac{1}{q}}.
\end{eqnarray*}
\end{theorem}

\begin{proof}
From Lemma \ref{lem 3.1} and using the power-mean inequality, we have%
\begin{eqnarray*}
&&\left\vert \frac{1}{6}\left[ f(a)+4f\left( \frac{2a+\eta (b,a)}{2}\right)
+f(a+\eta (b,a))\right] -\frac{1}{\eta (b,a)}\int_{a}^{a+\eta
(b,a)}f(x)dx\right\vert \\
&\leq &\eta \left( b,a\right) \\
&&\times \left\{ \left( \int_{0}^{\frac{1}{2}}\left\vert t-\frac{1}{6}%
\right\vert dt\right) ^{1-\frac{1}{q}}\left( \int_{0}^{\frac{1}{2}%
}\left\vert t-\frac{1}{6}\right\vert \left\vert f^{\prime }(a+t\eta \left(
b,a\right) )\right\vert ^{q}dt\right) ^{\frac{1}{q}}\right. \\
&&+\left. \left( \int_{\frac{1}{2}}^{1}\left\vert t-\frac{5}{6}\right\vert
dt\right) ^{1-\frac{1}{q}}\left( \int_{\frac{1}{2}}^{1}\left\vert t-\frac{5}{%
6}\right\vert \left\vert f^{\prime }(a+t\eta \left( b,a\right) )\right\vert
^{q}dt\right) ^{\frac{1}{q}}\right\} .
\end{eqnarray*}%
Since $\left\vert f^{\prime }\right\vert ^{q}$ is prequasiinvex function we
have$\ \ \ $

\begin{eqnarray*}
&&\int_{0}^{\frac{1}{2}}\left\vert t-\frac{1}{6}\right\vert \left\vert
f^{\prime }(a+t\eta \left( b,a\right) )\right\vert ^{q}dt \\
&\leq &\int_{0}^{\frac{1}{6}}\left( \frac{1}{6}-t\right) \left[ \max \left\{
\left\vert f^{\prime }(a)\right\vert ^{q},\left\vert f^{\prime
}(b)\right\vert ^{q}\right\} \right] dt \\
&&+\int_{\frac{1}{6}}^{\frac{1}{2}}\left( t-\frac{1}{6}\right) \left[ \max
\left\{ \left\vert f^{\prime }(a)\right\vert ^{q},\left\vert f^{\prime
}(b)\right\vert ^{q}\right\} \right] dt \\
&=&\frac{5}{72}\left[ \max \left\{ \left\vert f^{\prime }(a)\right\vert
^{q},\left\vert f^{\prime }(b)\right\vert ^{q}\right\} \right]
\end{eqnarray*}%
and 
\begin{eqnarray*}
&&\int_{\frac{1}{2}}^{1}\left\vert t-\frac{5}{6}\right\vert \left\vert
f^{\prime }(a+t\eta \left( b,a\right) )\right\vert ^{q}dt \\
&\leq &\int_{\frac{1}{2}}^{\frac{5}{6}}\left( \frac{5}{6}-t\right) \left[
\max \left\{ \left\vert f^{\prime }(a)\right\vert ^{q},\left\vert f^{\prime
}(b)\right\vert ^{q}\right\} \right] dt \\
&&+\int_{\frac{5}{6}}^{1}\left( t-\frac{5}{6}\right) \left[ \max \left\{
\left\vert f^{\prime }(a)\right\vert ^{q},\left\vert f^{\prime
}(b)\right\vert ^{q}\right\} \right] dt \\
&=&\frac{5}{72}\left[ \max \left\{ \left\vert f^{\prime }(a)\right\vert
^{q},\left\vert f^{\prime }(b)\right\vert ^{q}\right\} \right] .
\end{eqnarray*}%
From the above inequalities we get the desired result.
\end{proof}

\begin{corollary}
\label{co 4.1} In Theorem \ref{teo 4.1}, if we choose $q=1$ we obtain
\end{corollary}

\begin{eqnarray*}
&&\left\vert \frac{1}{6}\left[ f(a)+4f\left( \frac{2a+\eta (b,a)}{2}\right)
+f(a+\eta (b,a))\right] -\frac{1}{\eta (b,a)}\int_{a}^{a+\eta
(b,a)}f(x)dx\right\vert \\
&\leq &\frac{5}{36}\eta \left( b,a\right) \max \left\{ \left\vert f^{\prime
}(a)\right\vert ,\left\vert f^{\prime }(b)\right\vert \right\} .
\end{eqnarray*}

\begin{corollary}
\label{co 4.2} In Theorem \ref{teo 4.1}, if we choose $f(a)=f\left( \frac{%
2a+\eta (b,a)}{2}\right) =f(a+\eta (b,a))$ we obtain the midpoint type
inequality as follows:
\end{corollary}

\begin{eqnarray*}
&&\left\vert f\left( \frac{2a+\eta (b,a)}{2}\right) -\frac{1}{\eta (b,a)}%
\int_{a}^{a+\eta (b,a)}f(x)dx\right\vert \\
&\leq &\frac{5}{36}\eta \left( b,a\right) \max \left\{ \left\vert f^{\prime
}(a)\right\vert ,\left\vert f^{\prime }(b)\right\vert \right\} .
\end{eqnarray*}

\begin{theorem}
\label{teo 4.2} Let $I\subseteq 
\mathbb{R}
$ be an open invex subset with respect to $\eta :I\times I\rightarrow 
\mathbb{R}
_{+}$ and $f:I\rightarrow 
\mathbb{R}
$ be an absolutely continuous mapping on $I$,$a,b\in I$ with $\eta \left(
b,a\right) \neq 0.$ If $\left\vert f^{\prime }\right\vert $ is prequasiinvex
for some fixed $q>1$ then the following inequality holds%
\begin{eqnarray*}
&&\left\vert \frac{1}{6}\left[ f(a)+4f\left( \frac{2a+\eta (b,a)}{2}\right)
+f(a+\eta (b,a))\right] -\frac{1}{\eta (b,a)}\int_{a}^{a+\eta
(b,a)}f(x)dx\right\vert  \\
&\leq &2\eta (b,a)\left( \frac{1+2^{p+1}}{6^{p+1}(p+1)}\right) ^{\frac{1}{p}%
}\left( \frac{\max \left\{ \left\vert f^{\prime }(a)\right\vert
^{q},\left\vert f^{\prime }(b)\right\vert ^{q}\right\} }{2}\right) ^{\frac{1%
}{q}}
\end{eqnarray*}%
where $p=\frac{q}{q-1}.$
\end{theorem}

\begin{proof}
From Lemma \ref{lem 3.1} and using the H\"{o}lder inequality, we have 
\begin{eqnarray*}
&&\left\vert \frac{1}{6}\left[ f(a)+4f\left( \frac{2a+\eta (b,a)}{2}\right)
+f(a+\eta (b,a))\right] -\frac{1}{\eta (b,a)}\int_{a}^{a+\eta
(b,a)}f(x)dx\right\vert \\
&\leq &\eta (b,a)\left\{ \left( \int_{0}^{\frac{1}{2}}\left\vert t-\frac{1}{6%
}\right\vert ^{p}dt\right) ^{\frac{1}{p}}\left( \int_{0}^{\frac{1}{2}%
}\left\vert f^{\prime }(a+t\eta (b,a))\right\vert ^{q}dt\right) ^{\frac{1}{q}%
}\right. \\
&&+\left. \left( \int_{\frac{1}{2}}^{1}\left\vert t-\frac{5}{6}\right\vert
^{p}dt\right) ^{\frac{1}{p}}\left( \int_{\frac{1}{2}}^{1}\left\vert
f^{\prime }(a+t\eta (b,a))\right\vert ^{q}dt\right) ^{\frac{1}{q}}\right\} .
\end{eqnarray*}%
Since $\left\vert f^{\prime }\right\vert ^{q}$ is prequasiinvex, we obtain 
\begin{eqnarray*}
&&\left\vert \frac{1}{6}\left[ f(a)+4f\left( \frac{2a+\eta (b,a)}{2}\right)
+f(a+\eta (b,a))\right] -\frac{1}{\eta (b,a)}\int_{a}^{a+\eta
(b,a)}f(x)dx\right\vert \\
&\leq &\eta (b,a)\left\{ \left( \int_{0}^{\frac{1}{6}}\left( \frac{1}{6}%
-t\right) ^{p}dt+\int_{\frac{1}{6}}^{\frac{1}{2}}\left( t-\frac{1}{6}\right)
^{p}dt\right) ^{\frac{1}{p}}\right. \\
&&\times \left. \left( \int_{0}^{\frac{1}{2}}\max \left\{ \left\vert
f^{\prime }(a)\right\vert ^{q},\left\vert f^{\prime }(b)\right\vert
^{q}\right\} dt\right) ^{\frac{1}{q}}\right. \\
&&+\left. \left( \int_{\frac{1}{2}}^{\frac{5}{6}}\left( \frac{5}{6}-t\right)
^{p}dt+\int_{\frac{5}{6}}^{1}\left( t-\frac{5}{6}\right) ^{p}dt\right) ^{%
\frac{1}{p}}\right. \\
&&\times \left. \left( \int_{\frac{1}{2}}^{1}\max \left\{ \left\vert
f^{\prime }(a)\right\vert ^{q},\left\vert f^{\prime }(b)\right\vert
^{q}\right\} dt\right) ^{\frac{1}{q}}\right\} \\
&=&2\eta (b,a)\left( \frac{1+2^{p+1}}{6^{p+1}(p+1)}\right) ^{\frac{1}{p}%
}\left( \frac{\max \left\{ \left\vert f^{\prime }(a)\right\vert
^{q},\left\vert f^{\prime }(b)\right\vert ^{q}\right\} }{2}\right) ^{\frac{1%
}{q}}.
\end{eqnarray*}%
We used%
\begin{eqnarray*}
\int_{0}^{\frac{1}{6}}\left( \frac{1}{6}-t\right) ^{p}dt+\int_{\frac{1}{6}}^{%
\frac{1}{2}}\left( t-\frac{1}{6}\right) ^{p}dt &=&\int_{\frac{1}{2}}^{\frac{5%
}{6}}\left( \frac{5}{6}-t\right) ^{p}dt+\int_{\frac{5}{6}}^{1}\left( t-\frac{%
5}{6}\right) ^{p}dt \\
&=&\frac{1+2^{p+1}}{6^{p+1}(p+1)}
\end{eqnarray*}%
in the above inequality to complete the proof.
\end{proof}

\begin{theorem}
\label{teo 4.3} Under the assumptions of Theorem \ref{teo 4.2}, we have the
following inequality%
\begin{eqnarray*}
&&\left\vert \frac{1}{6}\left[ f(a)+4f\left( \frac{2a+\eta (b,a)}{2}\right)
+f(a+\eta (b,a))\right] -\frac{1}{\eta (b,a)}\int_{a}^{a+\eta
(b,a)}f(x)dx\right\vert \\
&\leq &\eta \left( b,a\right) \left( \frac{2(1+2^{p+1})}{6^{p+1}(p+1)}%
\right) ^{\frac{1}{p}}\left[ \frac{\max \left\{ \left\vert f^{\prime
}(a)\right\vert ^{q},\left\vert f^{\prime }(b)\right\vert ^{q}\right\} }{2}%
\right] ^{\frac{1}{q}}.
\end{eqnarray*}
\end{theorem}

\begin{proof}
From Lemma \ref{lem 3.1}, prequasiinvexity of $\left\vert f^{\prime
}\right\vert ^{q}$ and using the H\"{o}lder inequality, we have 
\begin{eqnarray*}
&&\left\vert \frac{1}{6}\left[ f(a)+4f\left( \frac{2a+\eta (b,a)}{2}\right)
+f(a+\eta (b,a))\right] -\frac{1}{\eta (b,a)}\int_{a}^{a+\eta
(b,a)}f(x)dx\right\vert \\
&\leq &\eta (b,a)\left[ \int_{0}^{1}\left\vert m(t)\right\vert \left\vert
f^{\prime }(a+t\eta (b,a))\right\vert dt\right] \\
&\leq &\eta (b,a)\left( \int_{0}^{1}\left\vert m(t)\right\vert ^{p}dt\right)
^{\frac{1}{p}}\left( \int_{0}^{1}\max \left\{ \left\vert f^{\prime
}(a)\right\vert ^{q},\left\vert f^{\prime }(b)\right\vert ^{q}\right\}
dt\right) ^{\frac{1}{q}} \\
&\leq &\eta (b,a)\left( \int_{0}^{\frac{1}{2}}\left\vert t-\frac{1}{6}%
\right\vert ^{p}dt+\int_{\frac{1}{2}}^{1}\left\vert t-\frac{5}{6}\right\vert
^{p}dt\right) ^{\frac{1}{p}}\left( \int_{0}^{1}\max \left\{ \left\vert
f^{\prime }(a)\right\vert ^{q},\left\vert f^{\prime }(b)\right\vert
^{q}\right\} dt\right) ^{\frac{1}{q}} \\
&=&\eta (b,a)\left( \frac{2(1+2^{p+1})}{6^{p+1}(p+1)}\right) ^{\frac{1}{p}}%
\left[ \frac{\max \left\{ \left\vert f^{\prime }(a)\right\vert
^{q},\left\vert f^{\prime }(b)\right\vert ^{q}\right\} }{2}\right] ^{\frac{1%
}{q}}
\end{eqnarray*}%
where we used the fact that%
\begin{equation*}
\int_{0}^{\frac{1}{2}}\left\vert t-\frac{1}{6}\right\vert ^{p}dt=\int_{\frac{%
1}{2}}^{1}\left\vert t-\frac{5}{6}\right\vert ^{p}dt=\frac{(1+2^{p+1})}{%
6^{p+1}(p+1)}.
\end{equation*}%
The proof is completed.
\end{proof}

\end{document}